\documentclass[11pt,reqno]{amsart}

\usepackage[T1]{fontenc}
\usepackage[margin=1.1in]{geometry}
\usepackage{amsmath,amssymb,amsthm,mathtools}
\usepackage{enumitem}
\usepackage{hyperref}
\usepackage{mathrsfs}
\usepackage{microtype}

\hypersetup{
  colorlinks=true,
  linkcolor=blue,
  citecolor=blue,
  urlcolor=blue
}

\newtheorem{theorem}{Theorem}[section]
\newtheorem{lemma}[theorem]{Lemma}
\newtheorem{proposition}[theorem]{Proposition}
\newtheorem{corollary}[theorem]{Corollary}
\newtheorem{question}[theorem]{Question}

\theoremstyle{definition}

\newtheorem{remark}[theorem]{Remark}
\newtheorem{example}[theorem]{Example}

\newcommand{\R}{\mathbb R}
\newcommand{\B}{\mathbb B}
\newcommand{\dd}{\,d}

\newcommand{\divf}{\operatorname{div}_f}
\newcommand{\Harm}{\mathcal H}
\newcommand{\Ric}{\operatorname{Ric}}
\newcommand{\Hess}{\operatorname{Hess}}
\newcommand{\tr}{\operatorname{tr}}
\newcommand{\Int}{\operatorname{Int}}
\newcommand{\Id}{\operatorname{Id}}

\title[Cohomology vanishing in Gaussian-weighted balls]
{Cohomology Vanishing for Free Boundary $f$-Minimal Submanifolds in Gaussian-Weighted Euclidean Balls}

\author{Niang Chen}
\address{Faculty of Arts and Sciences, Beijing Normal University, Zhuhai 519087, China}
\email{chenniang@bnu.edu.cn}

\begin{document}

\begin{abstract}
Let $M^n\subset \overline{\B_R^{n+k}}\subset \R^{n+k}$ be a compact orientable free boundary $f$-minimal submanifold of the Gaussian-weighted Euclidean ball $\left(\overline{\B_R^{n+k}},g_{\rm can},e^{-f}\dd V\right),
       f(x)=\frac c2 |x|^2,c\ge 0.$
We prove a cohomology vanishing theorem under the pointwise pinching condition $ |A|^2\le \frac{n-p}{R^2},1\le p<n.$
More precisely, the space of tangential $f$-harmonic $p$-forms vanishes, and hence$H^p(M;\R)=0.$
The proof is based on three elementary ingredients in the Gaussian-weighted ball: a weighted Hardy inequality obtained from the identity $\divf(x^T)=n-c|x|^2$, a cancellation in the weighted Weitzenb\"ock curvature operator, and a boundary reduction showing that tangential $f$-harmonic forms satisfy the same local absolute-boundary algebra as in the unweighted case.  The constant pinching threshold is independent of the Gaussian parameter $c$, and the argument also includes the unweighted case $c=0$; the strict interior positivity comes from the full Hardy--Weitzenb\"ock coefficient rather than from the sign of $c$ alone.
\end{abstract}


\maketitle

\section{Introduction}

A central theme in submanifold geometry is the use of analytic identities to extract topological information from curvature, boundary geometry, and extrinsic pinching.  In the closed setting this philosophy is classical: the Bochner--Weitzenb\"ock formula relates the Hodge Laplacian to curvature, and positivity of the curvature term forces harmonic forms, hence de Rham cohomology classes, to vanish.  We refer to Wu's monograph \cite{Wu1988Bochner} for the classical Bochner technique and its differential-geometric applications.  Closely related vanishing phenomena have also been studied through geometric measure theoretic methods, including Xin's application of integral currents to vanishing theorems \cite{Xin1984IntegralCurrents}.  When a boundary is present, the same circle of ideas must be combined with boundary integral identities; the integral-formula viewpoint developed by Yano \cite{Yano1970IntegralFormulas} is one of the classical sources behind the Reilly--Yano type formulas used in later work on differential forms.

For free boundary minimal submanifolds, the interaction between topology and geometry has been investigated through stability, index estimates, harmonic forms, and Bochner-type identities.  Index-topology comparisons were developed in closed and free boundary settings by Ambrozio--Carlotto--Sharp and related works \cite{ACS2018Closed,ACS2018FB,Fraser2007}; compactness and topology of embedded free boundary minimal surfaces were studied by Fraser--Li \cite{FraserLi2014}.  In Euclidean balls, cohomology and homology vanishing results are often proved by combining Reilly-type identities for differential forms with Hardy inequalities that control the boundary contribution; see, for example, Cavalcante--Mendes--Vit\'orio, Chen--Ge, and Onti--Vlachos \cite{CavalcanteMendesVitorio2019,ChenGe2022,OntiVlachos2022}.  For a recent survey of the broader sphere--ball correspondence, including Bochner--Hardy methods for free boundary cohomology vanishing, curvature pinching and gap phenomena, and index--topology comparisons, see \cite{Chen2026Survey}.

Weighted geometry provides a parallel variational framework in which the ambient measure is changed from $\dd V$ to $e^{-f}\dd V$.  Compact free boundary hypersurfaces in manifolds with density were studied by Castro--Rosales, including first and second variation formulas and rigidity results \cite{CastroRosales2014}; related compactness results for free boundary $f$-minimal surfaces in smooth metric measure spaces were obtained by Barbosa--Wei \cite{BarbosaWei2016}.  Weighted harmonic forms and weighted Hodge theory also play a role in index and topology questions for $f$-minimal hypersurfaces; see Bueler \cite{Bueler1999}, Seo--Yun \cite{SeoYun2020}, the index estimates of Impera--Rimoldi \cite{ImperaRimoldi2020} and Impera--Rimoldi--Savo \cite{ImperaRimoldiSavo2020}, and the free boundary $f$-minimal index estimates of Chen--Ge--Zhang \cite{ChenGeZhang2023}.  Throughout this paper, $f$ denotes the potential of the density $e^{-f}$, not the density itself.  Thus our Bakry--Emery convention is
\[
        \Ric_f=\Ric+\Hess f,
\]
which differs only notationally from papers that write the density as $e^{\psi}$ and hence use $\Ric_\psi=\Ric-\Hess\psi$.

The aim of this note is to isolate a simple Gaussian cancellation mechanism for free boundary $f$-minimal submanifolds in Euclidean balls.  Let
\[
        f(x)=\frac c2|x|^2,
        \qquad c\ge 0,
\]
and let
\[
        M^n\subset \overline{\B_R^{n+k}}\subset \R^{n+k}
\]
be compact, orientable, and free boundary $f$-minimal.  With the convention
\[
        \vec H=\sum_{i=1}^n A(e_i,e_i),
\]
the $f$-minimal equation reads
\[
        \vec H+(\bar\nabla f)^\perp=0.
\]
Since $\bar\nabla f=cx$, this becomes
\begin{equation}\label{eq:fminimal-basic}
        \vec H=-c x^\perp.
\end{equation}
Thus $c=0$ is precisely the unweighted free boundary minimal case.

The identity \eqref{eq:fminimal-basic} produces two cancellations.  First, in the weighted divergence of the tangential position vector one obtains
\begin{equation}\label{eq:weighted-div-position-intro}
        \divf(x^T)=n-c|x|^2,
\end{equation}
so the mean curvature term disappears from the Hardy calculation.  Second, in the weighted Weitzenb\"ock formula, the trace term in the extrinsic Bochner curvature is cancelled by the Hessian contribution of the Gaussian weight, leaving
\begin{equation}\label{eq:Bf-cancellation-intro}
        \mathcal B_f^{[p]}
        =cp\,\Id-\sum_\alpha (S_\alpha^{[p]})^2.
\end{equation}
Consequently,
\[
        \langle \mathcal B_f^{[p]}\omega,\omega\rangle
        \ge p(c-|A|^2)|\omega|^2.
\]
Compared with the unweighted free boundary vanishing theorems in Euclidean balls \cite{ChenGe2022}, the Gaussian drift has two opposite effects: it creates a negative bulk term in the weighted Hardy identity, but it also cancels the trace term in the extrinsic Weitzenb\"ock operator.  These two effects combine to give a constant pinching threshold independent of the Gaussian parameter $c$.  A sharper reading of the Hardy term shows, however, that the final strict positivity does not require $c>0$: the proof naturally covers all $c\ge 0$, including the unweighted case $c=0$.  Positive $c$ should instead be viewed as providing an additional interior curvature allowance in the spatially relaxed form of the pinching condition; see Remark \ref{rem:spatial-relaxation}.

The proof has three concrete inputs.  The weighted Hardy identity starts from \eqref{eq:weighted-div-position-intro}; the weighted Weitzenb\"ock term is reduced to \eqref{eq:Bf-cancellation-intro}; and after applying the boundary reduction and the Kato inequality, the remaining zeroth-order coefficient satisfies the sharper bound
\[
        K(x)\ge \frac p{R^2}\bigl(n-p-R^2|A|^2\bigr)
        +\left(pc+\frac{p^2}{R^2}\right)
        \left(1-\frac{|x|^2}{R^2}\right).
\]
The pinching condition makes the first term nonnegative, while the second term is strictly positive in the interior of the ball for every $c\ge 0$ and $1\le p<n$.

The main result is the following.

\begin{theorem}\label{thm:main}
Let $M^n\subset \overline{\B_R^{n+k}}$ be a smooth compact orientable free boundary $f$-minimal submanifold of the weighted Euclidean ball
\[
        \left(\overline{\B_R^{n+k}},g_{\rm can},e^{-f}\dd V\right)
        \subset
        \left(\R^{n+k},g_{\rm can},e^{-f}\dd V\right),
        \qquad f(x)=\frac c2|x|^2,
        \qquad c\ge 0.
\]
Assume that $\partial M\ne\emptyset$ and $\Int(M)\subset \B_R^{n+k}$.  Let $1\le p<n$.  If
\begin{equation}\label{eq:main-pinching}
        |A|^2\le \frac{n-p}{R^2}
        \qquad\text{on }M,
\end{equation}
then
\[
        \Harm^p_{N,f}(M)=0.
\]
In particular, by weighted Hodge theory with absolute boundary condition,
\[
        H^p(M;\R)=0.
\]
\end{theorem}

\begin{remark}\label{rem:p-range}
The range $1\le p<n$ is the natural range produced by the argument.  If one wants a statement closer to the classical pinching theorems, one may restrict to $1\le p\le \lfloor n/2\rfloor$ and then use duality separately for relative cohomology.  The present theorem concerns absolute cohomology and tangential $f$-harmonic forms.
\end{remark}

In dimension two, this gives the following topological consequence.

\begin{corollary}\label{cor:surface}
Let $\Sigma^2\subset \overline{\B_R^{2+k}}$ be a compact connected orientable free boundary $f$-minimal surface, where $f(x)=c|x|^2/2$ and $c\ge 0$.  If
\[
        |A|^2\le \frac1{R^2},
\]
then $\Sigma$ is homeomorphic to the disk $D^2$.
\end{corollary}

\begin{proof}
Theorem \ref{thm:main} gives $H^1(\Sigma;\R)=0$.  If an orientable compact connected surface has genus $g$ and $b$ boundary components, then $b_1=2g+b-1$.  Hence $0=2g+b-1$, which forces $g=0$ and $b=1$.  Thus $\Sigma$ is a disk.
\end{proof}

\section{Preliminaries}

We collect the notation and sign conventions used throughout the paper.  Let
\[
        M^n\subset \R^{n+k}
\]
be a compact orientable submanifold with boundary.  We denote by $\bar\nabla$ the Euclidean connection and by $\nabla$ the Levi-Civita connection of $M$.  The second fundamental form is
\[
        A(X,Y)=\bar\nabla_XY-\nabla_XY,
\]
and the non-normalized mean curvature vector is
\[
        \vec H=\tr A=\sum_{i=1}^n A(e_i,e_i).
\]
We use this non-normalized mean curvature vector throughout the proof.
Unless otherwise stated, $\{e_i\}_{i=1}^n$ denotes a local orthonormal tangent frame and $\{\nu_\alpha\}_{\alpha=1}^k$ denotes a local orthonormal normal frame.  We write
\[
        A_\alpha(X,Y)=\langle A(X,Y),\nu_\alpha\rangle,
        \qquad
        \langle S_\alpha X,Y\rangle=A_\alpha(X,Y),
\]
where $S_\alpha$ is the shape operator in the normal direction $\nu_\alpha$.  Set
\[
        h_\alpha=\tr S_\alpha=\langle \vec H,\nu_\alpha\rangle.
\]
Then
\[
        |A|^2=\sum_\alpha |S_\alpha|^2.
\]

Let $f\in C^\infty(\R^{n+k})$ be the potential of the weighted measure $e^{-f}\dd V$.  For a vector field $X$ tangent to $M$, the weighted divergence is
\[
        \divf X=\operatorname{div}X-\langle \nabla^M f,X\rangle.
\]
Equivalently, it is characterized by the weighted integration-by-parts formula
\begin{equation}\label{eq:weighted-div-thm}
        \int_M \divf X\,e^{-f}\dd V_M
        =\int_{\partial M}\langle X,\eta\rangle e^{-f}\dd V_{\partial M},
\end{equation}
where $\eta$ is the outward unit conormal of $\partial M$ in $M$.  This convention is consistent with the classical integral-formula framework for Riemannian geometry \cite{Yano1970IntegralFormulas}, with the only change being the weighted measure and the drift term $-\langle \nabla^Mf,X\rangle$.

For differential forms we use the weighted codifferential
\[
        \delta_f=\delta+i_{\nabla^M f}
\]
and the weighted Hodge Laplacian
\[
        \Delta_f=d\delta_f+\delta_f d.
\]
With this convention, on functions,
\[
        \Delta_f=-\divf\nabla.
\]
The Bakry--Emery tensor on $M$ is
\[
        \Ric_f=\Ric_M+\Hess_M f.
\]
The weighted Bochner--Weitzenb\"ock formula for an $f$-harmonic $p$-form $\omega$, i.e. a form satisfying $d\omega=0$ and $\delta_f\omega=0$, is written as
\begin{equation}\label{eq:weighted-bochner}
        -\frac12\Delta_f |\omega|^2
        =|\nabla\omega|^2+\langle \mathcal B_f^{[p]}\omega,\omega\rangle.
\end{equation}
The unweighted form of this identity is the analytic core of the Bochner technique \cite{Wu1988Bochner}; in the weighted setting, the curvature term is modified by the Hessian of the potential.

We denote the space of tangential $f$-harmonic $p$-forms by
\begin{equation}\label{eq:Hnf-def}
\Harm^p_{N,f}(M)
=
\{\omega\in \Omega^p(M): d\omega=0,\ \delta_f\omega=0\text{ in }M,
\ i_\eta\omega=0\text{ on }\partial M\}.
\end{equation}
Equivalently, these are the weighted harmonic representatives satisfying the absolute boundary condition
\[
        i_\eta\omega=0,
        \qquad
        i_\eta d\omega=0
        \qquad\text{on }\partial M,
\]
where the second condition is automatic in \eqref{eq:Hnf-def} since $d\omega=0$.

\begin{proposition}[Weighted Hodge theorem with absolute boundary condition]\label{prop:weighted-hodge}
Let $M$ be a compact smooth Riemannian manifold with boundary and let $f\in C^\infty(M)$.  Then each absolute de Rham cohomology class has a unique representative $\omega$ satisfying
\[
        d\omega=0,
        \qquad
        \delta_f\omega=0,
        \qquad
        i_\eta\omega=0
        \quad\text{on }\partial M.
\]
Consequently,
\begin{equation}\label{eq:weighted-hodge-iso}
        \Harm^p_{N,f}(M)\cong H^p(M;\R).
\end{equation}
\end{proposition}

\begin{proof}
The identity $\delta_f=e^f\delta e^{-f}$ shows that $\delta_f$ is the formal adjoint of $d$ with respect to the weighted inner product $\int_M\langle\cdot,\cdot\rangle e^{-f}$.  Since $e^{-f}$ is a positive smooth density and $M$ is compact, this weighted inner product is equivalent to the usual $L^2$ inner product.  The operator $\Delta_f=d\delta_f+\delta_f d$, together with the absolute boundary conditions $i_\eta\omega=0$ and $i_\eta d\omega=0$, is therefore a self-adjoint elliptic boundary problem with the same principal symbol and boundary symbol as the ordinary Hodge Laplacian.  Standard Hodge decomposition for elliptic boundary problems gives one and only one weighted harmonic representative in each absolute de Rham cohomology class; see the usual unweighted theory and its weighted analogues \cite{Bueler1999,IwaniecScottStroffolini1999,ChenGeZhang2023}.
\end{proof}

A compact submanifold $M^n\subset \overline{\B_R^{n+k}}$ is called a free boundary submanifold if
\[
        \partial M\subset \partial \B_R^{n+k},
        \qquad
        \Int(M)\subset \B_R^{n+k},
\]
and $M$ meets $\partial\B_R^{n+k}$ orthogonally along $\partial M$.  We use $\eta$ for the outward conormal of $\partial M\subset M$.  Thus, on the boundary of the Euclidean ball,
\begin{equation}\label{eq:free-boundary-xT}
        \eta=\frac{x}{R},
        \qquad
        x^T=R\eta
        \qquad\text{on }\partial M.
\end{equation}
This is opposite to the inward unit normal of the ambient ball used in some free boundary variational formulas; compare, for instance, the convention in \cite{ChenGeZhang2023}.  Here $x$ is the position vector and $x^T$ is its tangential projection onto $M$.

Finally, for the Gaussian potential
\[
        f(x)=\frac c2 |x|^2,
        \qquad c\ge 0,
\]
the $f$-minimal equation is
\[
        \vec H+(\bar\nabla f)^\perp=0.
\]
Since $\bar\nabla f=cx$, this gives
\begin{equation}\label{eq:H-vector-identity}
        \vec H=-cx^\perp.
\end{equation}
This identity is the source of the two cancellations used in the proof of Theorem \ref{thm:main}.

\section{A weighted boundary reduction}

The following lemma is the boundary input needed in the proof.  It says that, for tangential $f$-harmonic forms in a Gaussian-weighted ball, the local boundary term in the weighted Bochner formula is the same as the unweighted absolute-boundary term.

\begin{lemma}[Weighted boundary reduction]\label{lem:boundary-reduction}
Let $M^n\subset \overline{\B_R^{n+k}}$ be a compact free boundary submanifold and let
\[
        f(x)=\frac c2|x|^2,
        \qquad c\ge 0.
\]
If $\omega\in \Harm^p_{N,f}(M)$, then
\begin{equation}\label{eq:boundary-normal-derivative}
        \frac12\partial_\eta |\omega|^2
        =-\frac{p}{R}|\omega|^2
        \qquad\text{on }\partial M.
\end{equation}
Consequently,
\begin{equation}\label{eq:reilly-weighted}
        0
        =\frac{p}{R}\int_{\partial M}|\omega|^2e^{-f}
        +\int_M |\nabla\omega|^2e^{-f}
        +\int_M\langle \mathcal B_f^{[p]}\omega,\omega\rangle e^{-f}.
\end{equation}
\end{lemma}

\begin{proof}
We first prove the pointwise boundary identity.  Fix a point $q\in \partial M$.  Choose a local orthonormal frame $e_1,\ldots,e_{n-1}$ tangent to $\partial M$ near $q$ and set $e_n=\eta$.  Because $M$ meets $\partial\B_R$ orthogonally and $\eta$ is the outward conormal of $M$,
\[
        \eta=\frac{x}{R}\quad\text{on }\partial M.
\]
Therefore, for every vector $Y$ tangent to $\partial M$,
\begin{equation}\label{eq:boundary-shape}
        \nabla_Y^M\eta=\frac1R Y.
\end{equation}
Thus $\partial M\subset M$ is totally umbilical with principal curvatures $1/R$ with respect to $\eta$.

Since $\omega\in\Harm^p_{N,f}(M)$, we have $i_\eta\omega=0$ on $\partial M$ and $d\omega=0$ in $M$.  In particular,
\[
        i_\eta d\omega=0\quad\text{on }\partial M.
\]
The following calculation is pointwise at $q$.  For tangential indices $I=(i_1,\ldots,i_p)$, the identity $d\omega(\eta,e_{i_1},\ldots,e_{i_p})=0$ gives
\begin{align*}
0
&=(d\omega)(\eta,e_{i_1},\ldots,e_{i_p}) \\
&=(\nabla_\eta\omega)(e_{i_1},\ldots,e_{i_p})
  +\sum_{s=1}^p(-1)^s(\nabla_{e_{i_s}}\omega)(\eta,e_{i_1},\ldots,\widehat{e_{i_s}},\ldots,e_{i_p}).
\end{align*}
To evaluate the second term, differentiate the boundary identity
\[
        \omega(\eta,e_{i_1},\ldots,\widehat{e_{i_s}},\ldots,e_{i_p})=0
\]
along the tangential direction $e_{i_s}$.  The terms containing $\nabla^{\partial M}_{e_{i_s}}e_{i_j}$ still contain an $\eta$-slot and therefore vanish.  Hence
\[
        (\nabla_{e_{i_s}}\omega)(\eta,e_{i_1},\ldots,\widehat{e_{i_s}},\ldots,e_{i_p})
        =-\omega(\nabla^M_{e_{i_s}}\eta,e_{i_1},\ldots,\widehat{e_{i_s}},\ldots,e_{i_p}).
\]
Moving the first argument back to the $s$-th slot cancels the alternating sign in the exterior derivative.  Thus
\[
0=(\nabla_\eta\omega)(e_{i_1},\ldots,e_{i_p})
  +\sum_{s=1}^p \omega(e_{i_1},\ldots,\nabla^M_{e_{i_s}}\eta,\ldots,e_{i_p}).
\]
By \eqref{eq:boundary-shape}, this becomes
\[
        (\nabla_\eta\omega)(e_{i_1},\ldots,e_{i_p})
        =-\frac pR\omega(e_{i_1},\ldots,e_{i_p}).
\]
Since $\omega$ has no normal component at the boundary, this implies
\[
        \langle \nabla_\eta\omega,\omega\rangle
        =-\frac pR|\omega|^2,
\]
which is \eqref{eq:boundary-normal-derivative}.

We also record why the weighted boundary condition reduces to the usual local absolute-boundary algebra.  On $\partial M$, \eqref{eq:free-boundary-xT} gives
\[
        \nabla^M f=cx^T=cR\eta.
\]
Hence
\[
        i_{\nabla^M f}\omega=cR\,i_\eta\omega=0
        \qquad\text{on }\partial M.
\]
Since $\delta_f\omega=\delta\omega+i_{\nabla^M f}\omega=0$ holds smoothly up to the boundary, it follows that
\[
        \delta\omega=0\qquad\text{on }\partial M.
\]
Thus, pointwise on the boundary, $\omega$ satisfies the same local algebra as an unweighted harmonic form with absolute boundary condition, as in the classical Reilly--Yano formula for forms and its modern Reilly-type variants \cite{RaulotSavo2011,Yano1970IntegralFormulas,ChenGe2022}.

Now integrate the weighted Bochner formula \eqref{eq:weighted-bochner}.  Since our scalar convention is $\Delta_f=-\divf\nabla$, the weighted divergence theorem gives
\[
        \int_M -\Delta_f\left(\frac{|\omega|^2}{2}\right)e^{-f}
        =\int_{\partial M}\frac12\partial_\eta |\omega|^2 e^{-f}.
\]
Using \eqref{eq:boundary-normal-derivative}, we get
\[
        \int_M\left(|\nabla\omega|^2+
        \langle \mathcal B_f^{[p]}\omega,\omega\rangle\right)e^{-f}
        =-\frac pR\int_{\partial M}|\omega|^2e^{-f},
\]
which is \eqref{eq:reilly-weighted}.
\end{proof}

\section{A weighted Hardy inequality}

The next lemma is the weighted Hardy inequality adapted to the Gaussian free boundary setting.  It is the weighted analogue of Hardy-type inequalities for submanifolds and of the boundary Hardy inequalities used in unweighted free boundary cohomology vanishing theorems \cite{BatistaMirandolaVitorio2017,ChenGe2022}.

\begin{lemma}\label{lem:weighted-hardy}
Let $M^n\subset \overline{\B_R^{n+k}}$ be a compact free boundary $f$-minimal submanifold with
\[
        f(x)=\frac c2|x|^2,
        \qquad c\ge 0.
\]
Then for every $u\in H^1(M)$ and every $t>0$,
\begin{equation}\label{eq:weighted-hardy}
R\int_{\partial M}u^2e^{-f}
\ge
\int_M\left[-t|\nabla u|^2+
\left(n-c|x|^2-\frac1t|x^T|^2\right)u^2\right]e^{-f}.
\end{equation}
\end{lemma}

\begin{proof}
We first justify the computation for smooth $u$.  In Euclidean space,
\[
        \operatorname{div}(x^T)=n+\langle x,\vec H\rangle.
\]
Using the $f$-minimal identity \eqref{eq:H-vector-identity},
\[
        \operatorname{div}(x^T)=n-c|x^\perp|^2.
\]
Since $\nabla^M f=cx^T$,
\[
        \langle \nabla^M f,x^T\rangle=c|x^T|^2.
\]
Therefore
\begin{equation}\label{eq:weighted-div-position}
        \divf(x^T)=n-c|x^\perp|^2-c|x^T|^2=n-c|x|^2.
\end{equation}
Applying \eqref{eq:weighted-div-thm} to the vector field $u^2x^T$, and using $\langle x^T,\eta\rangle=R$ on $\partial M$, gives
\[
R\int_{\partial M}u^2e^{-f}
=
\int_M\left(2u\langle \nabla u,x^T\rangle+(n-c|x|^2)u^2\right)e^{-f}.
\]
Finally,
\[
        2u\langle \nabla u,x^T\rangle
        \ge -t|\nabla u|^2-\frac1t|x^T|^2u^2,
\]
and \eqref{eq:weighted-hardy} follows for smooth $u$.  For $u\in H^1(M)$, choose $u_j\in C^\infty(M)$ with $u_j\to u$ in $H^1(M)$.  The trace map $H^1(M)\to L^2(\partial M)$ is continuous, and the coefficients $|x|^2$, $|x^T|^2$, and $e^{-f}$ are smooth and bounded on the compact manifold $M$.  Therefore every term in the inequality passes to the limit, proving the $H^1$ statement.
\end{proof}

\section{The weighted Weitzenb\"ock curvature operator}

We now record the cancellation in the $f$-Bochner curvature term.  Related extrinsic Weitzenb\"ock lower bounds underlie Hodge--Laplacian estimates and rigidity results for submanifolds \cite{CuiSun2019,RaulotSavo2011,Savo2014}.  For a symmetric endomorphism $S:TM\to TM$, let $S^{[p]}$ be its natural action on $p$-forms, defined by
\[
(S^{[p]}\omega)(X_1,\ldots,X_p)
=
\sum_{j=1}^p \omega(X_1,\ldots,SX_j,\ldots,X_p).
\]
If $S$ is diagonal with eigenvalues $\lambda_1,\ldots,\lambda_n$, then $S^{[p]}$ is diagonal on the basis $e_{i_1}^*\wedge\cdots\wedge e_{i_p}^*$ with eigenvalues $\lambda_{i_1}+\cdots+\lambda_{i_p}$.

\begin{lemma}\label{lem:Bf-bound}
Let $M^n\subset \R^{n+k}$ be $f$-minimal for
\[
        f(x)=\frac c2|x|^2.
\]
Then the weighted Weitzenb\"ock curvature operator on $p$-forms satisfies
\begin{equation}\label{eq:Bf-exact}
        \mathcal B_f^{[p]}
        =cp\,\Id-\sum_\alpha (S_\alpha^{[p]})^2.
\end{equation}
In particular, for every $p$-form $\omega$,
\begin{equation}\label{eq:Bf-lower}
        \langle \mathcal B_f^{[p]}\omega,\omega\rangle
        \ge p(c-|A|^2)|\omega|^2.
\end{equation}
\end{lemma}

\begin{proof}
For a Euclidean submanifold, the ordinary Bochner curvature term is entirely extrinsic, because the ambient Euclidean curvature vanishes.  The Gauss-type splitting for an isometric immersion gives the pointwise formula
\begin{equation}\label{eq:extrinsic-bochner}
        \mathcal B_{\rm ext}^{[p]}
        =\sum_\alpha\left(h_\alpha S_\alpha^{[p]}-(S_\alpha^{[p]})^2\right).
\end{equation}
Equivalently, for a fixed normal direction one may diagonalize $S_\alpha$ at the point.  If $S_\alpha e_i=\lambda_i e_i$ and $I=\{i_1<\cdots<i_p\}$, then $S_\alpha^{[p]}$ has eigenvalue $\lambda_{i_1}+\cdots+\lambda_{i_p}$ on $e^{i_1}\wedge\cdots\wedge e^{i_p}$, and the contribution to the Bochner term is
\[
        h_\alpha(\lambda_{i_1}+\cdots+\lambda_{i_p})
        -(\lambda_{i_1}+\cdots+\lambda_{i_p})^2.
\]
Summing over an orthonormal normal frame gives \eqref{eq:extrinsic-bochner}; no flatness of the normal bundle or simultaneous diagonalization of different $S_\alpha$ is required \cite{Savo2014,ChenGe2022}.

We next identify the weighted zero-order term.  Since
\[
        \delta_f=\delta+i_{\nabla^M f},
\]
Cartan's formula gives
\[
        \Delta_f=d\delta_f+\delta_f d
        =\Delta+\mathcal L_{\nabla^M f}.
\]
For a $p$-form,
\[
        \mathcal L_{\nabla^M f}\omega
        =\nabla_{\nabla^M f}\omega+\Hess_M^{[p]}f(\omega).
\]
Thus the first-order term $\nabla_{\nabla^M f}$ is part of the weighted rough Laplacian, and the zero-order curvature term in the weighted Weitzenb\"ock formula is
\begin{equation}\label{eq:Bf-Bext-hess}
        \mathcal B_f^{[p]}=\mathcal B_{\rm ext}^{[p]}+\Hess_M^{[p]}f.
\end{equation}
This is the sign convention corresponding to $\delta_f=\delta+i_{\nabla^M f}$.

It remains to compute $\Hess_M^{[p]}f$.  For tangent vectors $X,Y$,
\[
        \Hess_M f(X,Y)
        =c\langle X,Y\rangle+\langle cx^\perp,A(X,Y)\rangle.
\]
Using $cx^\perp=-\vec H=-\sum_\alpha h_\alpha\nu_\alpha$, this becomes
\[
        \Hess_M f(X,Y)
        =c\langle X,Y\rangle-\sum_\alpha h_\alpha\langle S_\alpha X,Y\rangle.
\]
Since the identity endomorphism acts on $p$-forms as $p\,\Id$, we get
\begin{equation}\label{eq:hess-p-action}
        \Hess_M^{[p]}f
        =cp\,\Id-\sum_\alpha h_\alpha S_\alpha^{[p]}.
\end{equation}
Adding \eqref{eq:extrinsic-bochner} and \eqref{eq:hess-p-action} in \eqref{eq:Bf-Bext-hess} gives \eqref{eq:Bf-exact}.

It remains to estimate the last term.  For a symmetric endomorphism $S$, diagonalizing $S$ gives
\[
        |S^{[p]}\omega|^2\le p|S|^2|\omega|^2.
\]
Indeed, each eigenvalue of $S^{[p]}$ is a sum of $p$ eigenvalues of $S$, and
\[
        (\lambda_{i_1}+\cdots+\lambda_{i_p})^2
        \le p(\lambda_{i_1}^2+\cdots+\lambda_{i_p}^2)
        \le p|S|^2.
\]
Therefore
\[
        \sum_\alpha |S_\alpha^{[p]}\omega|^2
        \le p\sum_\alpha |S_\alpha|^2|\omega|^2
        =p|A|^2|\omega|^2.
\]
Together with \eqref{eq:Bf-exact}, this proves \eqref{eq:Bf-lower}.
\end{proof}

\section{Proof of the main theorem}

We now prove Theorem \ref{thm:main}.  Let
\[
        \omega\in \Harm^p_{N,f}(M)
\]
and set
\[
        u=|\omega|.
\]
By the weak Kato inequality, $u\in H^1(M)$ and
\begin{equation}\label{eq:kato}
        |\nabla u|\le |\nabla\omega|
        \qquad\text{a.e. on }M.
\end{equation}
Combining Lemma \ref{lem:boundary-reduction}, Lemma \ref{lem:Bf-bound}, and \eqref{eq:kato}, we get
\begin{equation}\label{eq:start-main-proof}
0
\ge
\frac pR\int_{\partial M}u^2e^{-f}
+
\int_M |\nabla u|^2e^{-f}
+
\int_M p(c-|A|^2)u^2e^{-f}.
\end{equation}

By Lemma \ref{lem:weighted-hardy}, for any $t>0$,
\begin{equation}\label{eq:hardy-insert}
\frac pR\int_{\partial M}u^2e^{-f}
\ge
\int_M\left[-\frac{pt}{R^2}|\nabla u|^2
+\frac p{R^2}\left(n-c|x|^2-\frac1t|x^T|^2\right)u^2\right]e^{-f}.
\end{equation}
The right-hand side of \eqref{eq:start-main-proof} is bounded below by the expression obtained by replacing the boundary term with the right-hand side of \eqref{eq:hardy-insert}.  Hence \eqref{eq:start-main-proof} implies
\begin{align}\label{eq:assembled}
0
\ge{}&
\int_M\left(1-\frac{pt}{R^2}\right)|\nabla u|^2e^{-f}
\nonumber\\
&+
\int_M\left[
        p(c-|A|^2)
        +\frac p{R^2}\left(n-c|x|^2-\frac1t|x^T|^2\right)
\right]u^2e^{-f}.
\end{align}
Choose
\begin{equation}\label{eq:t-choice}
        t=\frac{R^2}{p}.
\end{equation}
Then the gradient term disappears.  Hence
\begin{equation}\label{eq:K-integral}
        0\ge \int_M K(x)u^2e^{-f},
\end{equation}
where
\begin{equation}\label{eq:K-def}
        K(x)=p(c-|A|^2)+\frac p{R^2}\left(n-c|x|^2-\frac p{R^2}|x^T|^2\right).
\end{equation}
Using the sharper estimate
\[
        |x^T|^2\le |x|^2,
\]
we have the pointwise lower bound
\begin{align}\label{eq:K-lower}
K(x)
&\ge p(c-|A|^2)
+\frac p{R^2}\left(n-c|x|^2-\frac p{R^2}|x|^2\right) \nonumber\\
&=\frac p{R^2}\bigl(n-p-R^2|A|^2\bigr)
+\left(pc+\frac{p^2}{R^2}\right)
\left(1-\frac{|x|^2}{R^2}\right).
\end{align}
Under the pinching assumption \eqref{eq:main-pinching}, the first term on the right-hand side of \eqref{eq:K-lower} is nonnegative.  Moreover, because $c\ge 0$, $p>0$, and $\Int(M)\subset \B_R^{n+k}$,
\[
        \left(pc+\frac{p^2}{R^2}\right)
        \left(1-\frac{|x|^2}{R^2}\right)>0
        \qquad\text{on }\Int(M).
\]
Therefore $K(x)>0$ on $\Int(M)$.  The integrand $K u^2 e^{-f}$ is nonnegative and must vanish a.e. in $M$ by \eqref{eq:K-integral}.  Since $\partial M$ has zero $n$-dimensional measure and $K>0$ in the interior, $u=0$ a.e. in $\Int(M)$.  By smoothness of $\omega$, this gives $u\equiv0$ on $M$, and hence $\omega=0$.

Since every element of $\Harm^p_{N,f}(M)$ vanishes, we have
\[
        \Harm^p_{N,f}(M)=0.
\]
By the weighted Hodge isomorphism \eqref{eq:weighted-hodge-iso},
\[
        H^p(M;\R)=0.
\]
This proves Theorem \ref{thm:main}.

\section{Examples and equality comments}

\begin{example}[Flat equatorial disks]
Let
\[
        M=\overline{\B_R^n}\times\{0\}\subset \overline{\B_R^{n+k}}.
\]
Then $x^\perp=0$, $A=0$, and $\vec H=0$, so $M$ is free boundary and $f$-minimal for every quadratic potential $f(x)=c|x|^2/2$ with $c\ge 0$.  The pinching condition in Theorem \ref{thm:main} is satisfied for all $1\le p<n$, and the conclusion $H^p(M;\R)=0$ agrees with the topology of the disk.
\end{example}

\begin{remark}[On equality and the role of $c$]
The theorem does not assert a geometric classification in the equality case of the pinching condition.  The final strict positivity in the proof comes from the full coefficient
\[
        \left(pc+\frac{p^2}{R^2}\right)
        \left(1-\frac{|x|^2}{R^2}\right),
\]
which is positive in the interior of the Euclidean ball for all $c\ge 0$ and vanishes only on the boundary.  Thus the proof includes the unweighted case $c=0$ without a separate limiting argument.  The theorem therefore explains the equality case of the proof, but it does not prove that the constant $(n-p)/R^2$ is sharp; the sharpness of this constant threshold remains open.
\end{remark}

\begin{remark}[Spatial relaxation of the pinching]\label{rem:spatial-relaxation}
The proof uses positivity of the coefficient $K(x)$ in \eqref{eq:K-def}.  The sharper lower bound \eqref{eq:K-lower} shows that the constant pinching assumption may be replaced by the pointwise strict interior condition
\[
        |A|^2
        < \frac{n-p}{R^2}
        +\left(c+\frac p{R^2}\right)
        \left(1-\frac{|x|^2}{R^2}\right)
        \qquad\text{on }\Int(M).
\]
Under this hypothesis the same proof gives $K(x)>0$ on $\Int(M)$ and hence $\Harm^p_{N,f}(M)=0$.  The non-strict version of this spatially relaxed inequality gives only $K\ge 0$ from this argument; a full equality-case statement for that more flexible condition would require additional analysis, for example a suitable unique-continuation or rigidity argument.  Geometrically, the term involving $c$ shows that a positive Gaussian parameter allows more extrinsic curvature away from the boundary, while the term $p/R^2$ is already present in the unweighted case.
\end{remark}

\begin{remark}[Non-flat examples]
The flat equatorial disk shows compatibility of the theorem with the basic topology, but it does not test the optimality of the constant.   Finding non-totally geodesic Gaussian free boundary $f$-minimal examples at or below the threshold, or proving that none exist under additional hypotheses, would clarify the rigidity content of the theorem.
\end{remark}

\section{Remarks and open questions}

\begin{remark}[Formal independence of the Gaussian parameter]
The constant pinching threshold in Theorem \ref{thm:main},
\[
        |A|^2\le \frac{n-p}{R^2},
\]
is formally independent of the Gaussian parameter $c$.  This is a consequence of two cancellations: the identity
\[
        \divf(x^T)=n-c|x|^2
\]
in the Hardy inequality, and the identity
\[
        \mathcal B_f^{[p]}=cp\,\Id-\sum_\alpha(S_\alpha^{[p]})^2
\]
in the weighted Bochner term.  The independence of the constant threshold should not obscure the stronger pointwise information in \eqref{eq:K-lower}: before one freezes the estimate to a constant pinching condition, the proof contains the positive interior term
\[
        \left(pc+\frac{p^2}{R^2}\right)
        \left(1-\frac{|x|^2}{R^2}\right).
\]
The summand $p^2/R^2$ comes from using $|x^T|^2\le |x|^2$ rather than the cruder bound $|x^T|^2\le R^2$, and it is exactly what allows the proof to recover the unweighted case $c=0$.
\end{remark}

\begin{remark}[The traceless second fundamental form]
Let
\[
        \Phi=A-\frac1n\vec H\,g
\]
be the traceless second fundamental form.  Then
\[
        |A|^2=|\Phi|^2+\frac1n|\vec H|^2.
\]
Since $\vec H=-cx^\perp$, this becomes
\[
        |A|^2=|\Phi|^2+\frac{c^2}{n}|x^\perp|^2.
\]
Thus Theorem \ref{thm:main} can be rewritten as a pinching statement involving $|\Phi|^2$.  For large $c$, the constant condition $|A|^2\le(n-p)/R^2$ strongly constrains the normal component $x^\perp$; Remark \ref{rem:spatial-relaxation} records the corresponding pointwise interior allowance supplied by the Gaussian term.
\end{remark}

\begin{remark}[On a possible weighted refined Kato inequality]
The proof deliberately uses only the standard Kato inequality
\[
        |\nabla\omega|^2\ge |\nabla|\omega||^2.
\]
The classical refined Kato inequality for ordinary harmonic forms \cite{CalderbankGauduchonHerzlich2000} cannot be applied directly to $f$-harmonic forms, because
\[
        \delta_f\omega=0
        \qquad\not\Rightarrow\qquad
        \delta\omega=0
\]
in the interior.  Indeed,
\[
        \delta_f\omega=\delta\omega+i_{\nabla^M f}\omega=0
        \quad\Longrightarrow\quad
        \delta\omega=-i_{\nabla^M f}\omega=-c\,i_{x^T}\omega,
\]
which is generally a nonzero first-order drift constraint.  This is the obstruction to importing the usual trace-free orthogonal decomposition of $\nabla\omega$ without modification.

The following discussion is only formal unless such a weighted refined Kato inequality is established.  Suppose that one could prove a weighted refined Kato inequality of the form
\begin{equation}\label{eq:hypothetical-refined-kato}
        |\nabla\omega|^2\ge \beta_f |\nabla|\omega||^2,
        \qquad \beta_f>1,
\end{equation}
for forms satisfying $d\omega=0$ and $\delta_f\omega=0$.  Then the same argument would replace the gradient coefficient in \eqref{eq:assembled} by
\[
        \beta_f-\frac{pt}{R^2}.
\]
Choosing
\[
        t=\frac{\beta_f R^2}{p}
\]
would formally improve the threshold to
\[
        |A|^2\le \frac{n-p/\beta_f}{R^2}.
\]
This suggests the following natural problem.
\end{remark}

\begin{question}
Is there a refined Kato inequality adapted to weighted harmonic forms, i.e. to forms satisfying
\[
        d\omega=0,
        \qquad
        \delta_f\omega=0,
\]
with a constant $\beta_f>1$ under geometric assumptions relevant to Gaussian $f$-minimal submanifolds?
\end{question}

\begin{remark}[Scope of the theorem]
The theorem is stated for tangential $f$-harmonic forms and therefore for absolute cohomology.  The corresponding problem for normal $f$-harmonic forms is subtler, because the weighted boundary term contains the geometry of $\partial M\subset M$ and the normal derivative of the weight.  A separate treatment would be required for relative cohomology.
\end{remark}

\end{document}